\documentclass[bibtex,en]{elegantpaper}
\usepackage{extarrows}
\usepackage{mathrsfs}
\numberwithin{equation}{section}
\allowdisplaybreaks[4]

\everymath{\displaystyle}

\title{Liouville theorems for anisotropic $p$-Laplace equations with a semilinear term\footnotemark[2]}

\author{Weizhao Liang \and Tian Wu\footnotemark[1] \and Jin Yan}

\begin{document}

\date{}
\maketitle

\renewcommand{\thefootnote}{\fnsymbol{footnote}}

\footnotetext[1]{The corresponding author.}
\footnotetext[2]{Authors were funded by Anhui Postdoctoral Scientific Research Program Foundation (No. 2025B1055). Welcome to contact us:
\begin{itemize}
    \item Weizhao Liang, School of Mathematical Sciences, University of Science and Technology of China, Hefei, Anhui, 230026, People's Republic of China. Email: Lwz740@mail.ustc.edu.cn
    \item Tian Wu, School of Mathematical Sciences, University of Science and Technology of China, Hefei, Anhui, 230026, People's Republic of China. Email: wt1997@ustc.edu.cn
    \item Jin Yan,  School of Mathematical Sciences, University of Science and Technology of China, Hefei, Anhui, 230026, People's Republic of China. Email: yjoracle@mail.ustc.edu.cn
\end{itemize}}

\begin{abstract}

    In this paper, we investigate Liouville theorems for solutions to the anisotropic $p$-Laplace equation
    $$-\Delta_p^H u=-\operatorname{div}(a(\nabla u))=f(u),\quad\text{in }\mathbb{R}^n,$$
    where the semilinear term $f$ may be positive, negative, or sign-changing. When $f$ is positive (negative) and satisfies certain conditions, Serrin's technique is applied to show that every positive supersolution (subsolution) must be constant. For the subcritical case, we use the invariant tensor method to prove nonexistence results for positive solutions. In particular, by applying the differential identity established in the subcritical case to the critical case, we provide a simplified new proof of the classification of positive solutions to the critical case $f(u)=u^{p^*-1}$. For sign-changing solutions, every stable solution or solution that is stable outside a compact set is trivial under certain conditions on $f$.
\end{abstract}

\tableofcontents
\section{Introduction}\label{sec:introduction}
\setcounter{equation}{0}

Let $n\geqslant 2$ and $n\in\mathbb N^*$. Let $H\in C^2(\mathbb R^n\backslash\{0\})\cap C(\mathbb R^n)$ be a positively homogeneous function of degree one, which is positive on $\mathbb{S}^{n-1}$. By homogeneity, we see that $H$ can be naturally extended to a continuous function in $\mathbb R^n$ by setting $H(0)=0$. We also assume $H$ is uniformly convex, that is, the unit ball $B_1^H=\{\xi\in\mathbb R^n\mid H(\xi)<1\}$ is uniformly convex. Let $H_0$ be the dual function of $H$ defined by $H_{0}(x):=\sup\limits_{H(\xi)=1}x\cdot\xi$, which is also positively homogeneous of degree one. Set $\hat H(x)=H(-x)$, $\hat H_{0}(x)=H_{0}(-x)$, $\nabla_{\xi}H=(\nabla_{\xi_1}H,\cdots,\nabla_{\xi_n}H)$ and $\nabla H_{0}=(\nabla_{x_1}H_0,\cdots,\nabla_{x_n}H_0)$, with second derivatives defined similarly.

For $p>1$, we study the following anisotropic $p$-Laplace equation
\begin{equation}\label{eq:anisotropic} 
    -\Delta_p^H u=-\operatorname{div}(a(\nabla u))=f(u),\quad\text{in }\mathbb R^n,
\end{equation}
where $a(\xi):=\frac{1}{p}\nabla_{\xi} H^p(\xi)$ and $f$ is sufficiently smooth. When $H$ is the Euclidean norm, the anisotropic $p$-Laplacian $\Delta_p^H$ reduces to the classical $p$-Laplacian $\Delta_p$. Denote $p_*=\frac{(n-1)p}{n-p}$ and $p^*=\frac{np}{n-p}$.

For the classical $p$-Laplacian $\Delta_p$, there are many results concerning the classification of positive solutions to the equation
\begin{equation}\label{eq:classical-p-Laplacian}
    -\Delta_p u=u^{\alpha},\quad\text{in }\mathbb R^n.
\end{equation}
Gidas-Spruck \cite{GS1981} proved that for $p=2$ and $1<\alpha<\frac{n+2}{n-2}$ (the subcritical case), every nonnegative solution to \eqref{eq:classical-p-Laplacian} must be trivial, that is $u\equiv 0$. Caffarelli-Gidas-Spruck \cite{CGS1989} classified positive solutions for $p=2$ and $\alpha=\frac{n+2}{n-2}$ (the critical case) using the moving plane method. Subsequently, Chen-Li \cite{CL1991} provided a simpler and more elementary proof of the same results using the moving plane method.

While the method of moving planes has played a central role in classifications of positive solutions to semilinear elliptic equations, it relies heavily on the maximum principle and certain symmetry properties of the underlying space. In situations where these tools are not directly applicable, alternative methods based on vector fields and differential identities have been developed to derive classification and rigidity results.

The method of differential identities originates from the work of Obata \cite{Obata1971}, where he used a Bochner-type formula to derive a differential identity satisfied by positive solutions to the Yamabe equation on closed manifolds:
$$(v^{1-n}E_{ij}v^j),^{~i}=v^{1-n}\sum_{i,j=1}^n |E_{ij}|^2+\mathcal{R},$$
where $\mathcal{R}\geqslant 0$, leading to the classification and rigidity results for solutions to the Yamabe equation.

For the case $1<p<n$, Serrin-Zou \cite{Serrin-Zou2002} established the corresponding differential identities and proved that the equation \eqref{eq:classical-p-Laplacian} admits no positive solutions when $p-1<\alpha<p^*-1$. Moreover, they showed that if there exist constants $p-1<\beta\leqslant\alpha<p^*-1$ such that the semilinear term $f$ satisfies
$$uf'(u)\leqslant \alpha f(u),\quad\forall u>0\quad\text{and}\quad f(u)\geqslant u^{\beta},\quad\text{for sufficiently large }u,$$
then every nonnegative solution to $-\Delta_p u=f(u)$ in $\mathbb R^n$ must be trivial.

For the critical case $\alpha=p^*-1$ with $p\ne 2$, under the finite energy condition
$$
u\in D^{1,p}(\mathbb R^n):= \left\{u\in L^{p^*}(\mathbb R^n)\mid\nabla u\in L^p(\mathbb R^n)\right\},$$
the following classification results for positive solutions to \eqref{eq:classical-p-Laplacian} have been established: Damascelli-Merch\'an-Montoro-Sciunzi \cite{Damascelli-2014} for $1<\frac{2n}{n+2}\leqslant p<2$, V\'etois \cite{Vetois-2016} for $1<p<2$, and Sciunzi \cite{Sciunzi-2016} for $p>2$. In all cases, positive solutions are radially symmetric and take the explicit form
\begin{equation}\label{sol:critical}
    u=u_{\mu,x_0}(x):=\left(\frac{n^{\frac{1}{p}}\left(\frac{n-p}{p-1}\right)^{\frac{p-1}{p}}\mu^{\frac{1}{p-1}}}{\mu^{\frac{p}{p-1}}+|x-x_0|^{\frac{p}{p-1}}}\right)^{\frac{n-p}{p}},\quad \mu>0,~x_0\in\mathbb R^n.
\end{equation}
When the finite energy condition is removed, the problem is still completely open. Recently, Catino-Monticelli-Roncoroni \cite{Catino-Monticelli-Roncoroni-2023} use the method of differential identities to obtain a result in dimensions $n=2,3$ for $\frac{n}{2}<p<2$. Using the same method, Ou \cite{ou2025classification} and V\'etois \cite{Vetois2024} showed that, for $\frac{n+1}{3}<p<n$ and $p_n<p<n$, respectively, every positive solution to \eqref{eq:classical-p-Laplacian} must also take the form \eqref{sol:critical}, where
$$p_n:=\begin{cases}
    \frac{8}{5},&\text{if }n=4,\\[8pt]
    \frac{4n+3-\sqrt{4n^2+12n-15}}{6},&\text{if }n\geqslant 5.
\end{cases}$$
Subsequently, for $p_n<p<n$, Sun-Wang \cite{Sun-Wang-2025} provided a classification of positive solutions to \eqref{eq:classical-p-Laplacian} in complete, connected and noncompact Riemannian manifold of dimension $n$ with nonnegative Ricci curvature, where
$$p_n:=\begin{cases}
    \frac{n^2}{3n-2},&\text{if }2\leqslant n\leqslant 4,\\[8pt]
    \frac{n^2+2}{3n},&\text{if }n\geqslant 5.
\end{cases}$$

In the setting of CR geometry, Jerison-Lee \cite{Jerison-Lee-1988} employed computer-assisted calculations to derive a differential identity and proved a classification theorem for finite energy solutions to the equation (with $\alpha = \frac{n+2}{n}$):
$$-\Delta_{b}u=u^{\alpha},\quad\text{in }\mathbb{H}^n.$$
Ma-Ou \cite{Ma-Ou-2023} generalized Jerison-Lee's identity to obtain a Liouville result for the subcritical case $1<\alpha<\frac{n+2}{n}$. In \cite{Ma-Wu-2024}, Ma and Wu, the second author of this article, first introduced the invariant tensor
$$E_{ij}=u_{ij}+c\frac{u_i u_j}{u}-\frac{1}{n}\left(\Delta u+c\frac{|\nabla u|^2}{u}\right)g_{ij},$$
and, in combination with differential identities, provided simplified proofs of both Gidas-Spruck's result and Ma-Ou's result. Furthermore, by using the invariant tensor technique, Ma-Ou-Wu \cite{Ma-Ou-Wu-2025} addressed a question raised by Jerison-Lee regarding whether there exists a theoretical framework that can predict the existence and structure of such formulae. Besides, Zhu \cite{zhu2024} first extended the invariant tensor technique to the classical $p$-Laplacian case. 

Now, we turn our attention to the anisotropic equation \eqref{eq:anisotropic} and introduce the definitions of solutions, supersolutions, and subsolutions in this setting.

\begin{definition}
    We say that $u$ is a supersolution (subsolution) of \eqref{eq:anisotropic} if $u\in W^{1,p}_{\mathrm{loc}}(\mathbb R^n)\cap C(\mathbb R^n)$ satisfies
    \begin{equation}\label{weak-supsol}
        \int_{\mathbb R^n}a(\nabla u)\cdot\nabla\varphi\geqslant(\leqslant)\int_{\mathbb R^n} f(u)\varphi,\quad\forall\varphi\in C_c^\infty(\mathbb R^n),~\varphi\geqslant 0.
    \end{equation}
    We say that $u$ is a solution to \eqref{eq:anisotropic} if $u\in W^{1,p}_{\mathrm{loc}}(\mathbb R^n)\cap L^\infty_{\mathrm{loc}}(\mathbb R^n)$ satisfies
    \begin{equation}\label{weak-sol}
        \int_{\mathbb R^n}a(\nabla u)\cdot\nabla\varphi=\int_{\mathbb R^n}f(u)\varphi,\quad\forall\varphi\in C_c^\infty(\mathbb R^n).
    \end{equation}
\end{definition}

\begin{remark}
    By a standard approximation argument, the test function $\varphi$ can be taken from $W^{1,p}_{\operatorname{loc}}(\mathbb R^n)$ with compact support.
\end{remark}

Our first theorem is as follows:

\begin{theorem}\label{thm:Serrin-1}
    Let $1<p<n$. If $f:\mathbb{R}_+\to\mathbb{R}_+$ is continuous and satisfies 
    $$\int_0^1[f(v)]^{-\frac{1}{p_*-1}}\mathrm{d}v<+\infty,$$ then the equation \eqref{eq:anisotropic} admits no positive supersolution.
\end{theorem}

\begin{remark}
    By Corollary \ref{cor:p>=n}, the equation \eqref{eq:anisotropic} admits no positive solutions if $p\geqslant n$ and $f\geqslant 0$.
\end{remark}

As a consequence of Theorem \ref{thm:Serrin-1}, the following Liouville result holds if $\liminf_{x\to0^+}f(x)>0$.

\begin{corollary}
    Let $1<p<n$. If $f:\mathbb{R}_+\to\mathbb{R}_+$ is continuous and satisfies $\liminf_{x\to0^+}f(x)>0$, then the equation \eqref{eq:anisotropic} admits no positive supersolutions.
\end{corollary}

For the special case $f(u)\geqslant cu^\alpha$, $c>0$, Theorem \ref{thm:Serrin-1} recovers the classical Serrin exponent result.

\begin{corollary}
If $1<p<n$ and $\alpha<p_*-1$, then $-\Delta_p^Hu=u^\alpha$ has no positive supersolutions. 
\end{corollary}

For the case $f<0$, we also have a Liouville theorem.

\begin{theorem}\label{thm:Serrin-2}
    Let $p>1$, and let $f:\mathbb{R}_+\to\mathbb{R}_-$ be continuous. Suppose one of the following conditions holds:
    
    (1) For $p<n$, there exists $q\in(p-1,p_*-1)$ such that $\int_1^{+\infty}[-f(v)]^{-\frac{1}{q}}\mathrm{d}v<\infty$;
    
    (2) For $p\geqslant n$, there exists $q>p-1$ such that $\int_1^{+\infty}[-f(v)]^{-\frac{1}{q}}\mathrm{d}v<\infty$.

    \noindent Then the equation \eqref{eq:anisotropic} admits no positive subsolutions. 
\end{theorem}

\begin{corollary}
    If $\alpha>p-1>0$, then $-\Delta_p^Hu=-u^\alpha$ admits no positive subsolutions. 
\end{corollary}

For the subcritical case, we employ the invariant tensor technique to extend Serrin-Zou's results to the anisotropic $p$-Laplace equation \eqref{eq:anisotropic}.

\begin{theorem}\label{thm:subcritical}
    Let $1<p<n$. Suppose $f\in C^1(\mathbb{R}_+)$ satisfies
    
    (1) There exist $\beta_0>p-1$, $c>0$ and $M>0$ such that $f(u)\geqslant cu^{\beta_0}$ for all $u\geqslant M$;
    
    (2) There exists $\alpha_0\in[\beta_0, p^*-1)$ such that $uf'(u)\leqslant\alpha_0 f(u)$ for all $u>0$.
    
    \noindent Then the equation \eqref{eq:anisotropic} admits no positive solutions.
\end{theorem}
 
Similar to the proof of Theorem \ref{thm:subcritical}, we obtain the following Liouville result for bounded solutions. Here, the semilinear term $f$ may change sign, and we lack the asymptotic behavior of the positive solution $u$ (see Lemma \ref{lem:asymptotic}). 

\begin{theorem}\label{cor:subcritical}
    Let $1<p<n$ and $f\in C^1(\mathbb{R}_+)$. If there exists $\alpha_0\in (p-1,p^*-1)$ such that 
    $$\frac{p_*-2}{p^*-1}\alpha_0\leqslant 2(p-1)\quad\text{and}\quad uf'(u)\leqslant\alpha_0 f(u),\quad\forall u>0.$$ Then the equation \eqref{eq:anisotropic} admits no bounded positive solutions.
\end{theorem}

\begin{remark}
    When $p=2$, the assumption $\frac{p_*-2}{p^*-1}\alpha_0\leqslant 2(p-1)$ always holds because of $n\geqslant3$. Hence the assumption $\frac{p_*-2}{p^*-1}\alpha_0\leqslant 2(p-1)$ can be removed in Laplacian case.
\end{remark}

For the critical case $f(u)=u^{p^*-1}$, classifications of finite-energy positive solutions to the following anisotropic critical equation were established by Ciraolo-Figalli-Roncoroni \cite{Figalli2020}:
\begin{equation}\label{eq:anisotropic-critical}
    \begin{cases}
        -\Delta_{p}^{H}u=u^{p^*-1},~u>0,&\text{in }\Sigma,\\
        a(\nabla u)\cdot\nu=0,&\text{on }\partial\Sigma, \\
        u\in\mathcal{D}^{1,p}(\Sigma), 
    \end{cases}
\end{equation}
where $\Sigma$ is defined in Theorem \ref{thm:critical}, $\nu$ is unit outer vector of $\partial\Sigma$, and $\mathcal{D}^{1,p}(\Sigma):=\{u\in L^{p^*}(\Sigma)\mid \nabla u\in L^p(\Sigma)\}$. Using the key differential identity (proven in the subcritical case), we provide a simplified proof of this classification result.

\begin{theorem}(Theorem 1.1 in \cite{Figalli2020})\label{thm:critical}
    Let $1<p<n$, and let $\Sigma=\mathbb{R}^k\times\mathcal{C}\subset \mathbb R^n$ be a convex cone, where $\mathcal{C} \subset \mathbb{R}^{n-k}$ is a convex cone which does not contain a line. Then every solution $u$ of \eqref{eq:anisotropic-critical} has the form
    \begin{equation}\label{sol:anisotropic-critical-sol}
        u\equiv u_{\mu,x_0}^H := \left( \frac{ n^{\frac{1}{p}}  \left(\frac{n-p}{p-1}\right)^{\frac{p-1}{p}} \mu^{\frac{1}{p-1}} }{\mu^\frac{p}{p-1} + \hat{H}_0(x-x_0)^\frac{p}{p-1}} \right)^{\frac{n-p}{p}},\quad \forall x\in \Sigma,
    \end{equation}
    for some $\mu>0 $ and $x_0\in \overline{\Sigma}$. Moreover, 
    
    (1) if $k=n$, then $\Sigma=\mathbb R^n$, and $x_0$ may be a generic point in $\mathbb R^n$;
    
    (2) if $k\in\{1,\dots,n-1\}$, then $x_0\in\mathbb{R}^k\times\{0\}$;
     
    (3) if $k=0$, then $x_0=0$.
\end{theorem}

As an application of Liouville theorems, we study the singular behavior of nonnegative solutions to the equation
\begin{equation}\label{eq:anisotropic in Omega}
    -\Delta_p^H u=f(u),\quad\text{in }\Omega,
\end{equation}
where $\Omega\subset\mathbb R^n$ is a domain.
By combining the rescaling argument with a key doubling property, we can derive an a priori estimate of possible singularities of local solutions.

\begin{theorem}\label{thm:asymptotic-1}
    Let $\Omega\subset\mathbb R^n$ be a domain, $0<p-1<\alpha_0<p^*-1$. If $f\in C([0,\infty))\cap C^1(\mathbb{R}_+)$ satisfies
    \begin{equation}\label{eq:asymptotic-f}
        \lim_{t\to\infty}t^{-\alpha_0}f(t)=\tau\in\mathbb{R}_+.
    \end{equation}
    Then there exists a positive constant $C=C(n,p,\tau)$ such that for any nonnegative solution $u$ to \eqref{eq:anisotropic in Omega}, there holds
    \begin{equation}\label{eq:asymptotic-estimate 1}
        u(x)^{\frac{\alpha_0-p+1}{p}}+|\nabla u(x)|^{\frac{\alpha_0-p+1}{\alpha_0+1}}\leqslant C\left(1+d^{-1}(x)\right),\quad\forall x\in\Omega,
    \end{equation}
    where $d(x)=\operatorname{dist}(x,\partial\Omega)$. In particular, if $\Omega=B_{R_0}\backslash\{0\}$ for some $R_0>0$, then
    \begin{equation*}
        u(x)^{\frac{\alpha_0-p+1}{p}}+|\nabla u(x)|^{\frac{\alpha_0-p+1}{\alpha_0+1}}\leqslant C\left(1+|x|^{-1}\right),\quad\forall x\in B_{\frac{R}{2}}\backslash\{0\}.
    \end{equation*}
\end{theorem}
\begin{remark}
    As stated in Corollary \ref{cor:asymptotic-2}, when $\Omega \ne \mathbb R^n$ and $f(u)=u^{\alpha_0}$, we obtain the more precise estimate:
    $$u(x)^{\frac{\alpha_0-p+1}{p}}+|\nabla u(x)|^{\frac{\alpha_0-p+1}{\alpha_0+1}}\leqslant C d^{-1}(x),\quad \forall x\in\Omega.$$
\end{remark}

Beyond positive (or nonnegative) solutions, a natural question is whether sign-changing solutions to the equation
\begin{equation}\label{eq:classical-p-Laplacian-2}
    -\Delta_p u=|u|^{\alpha-1}u,\quad\text{in } \mathbb R^n
\end{equation}
also satisfy Liouville theorems. For the case $p=2$, this question was partially addressed by Bahri-Lions \cite{Bahri-Lions-1992}, who employed integral inequalities and Pohozaev identity to show that, when $1<\alpha<2^*-1$, every solution with finite Morse index to the above equation must be trivial. Later, Farina \cite{Farina-2007} and Damascelli-Farina-Sciunzi-Valdinoci \cite{Damascelli-Farina-Sciunzi-2009} established Liouville theorems for stable solutions to \eqref{eq:classical-p-Laplacian-2} in the cases $p=2$ and $p>2$, respectively, provided that $1<\alpha<p_{\mathrm{JL}}$ (see Corollary \ref{cor:stable-p-Laplacian}). It should be noted that $p_{\mathrm{JL}}>p^*-1$, when $n>2$. Recently, Farina \cite{Farina2024-B-anisotropic-stable} studied the more general quasilinear equation $-\Delta_B^H u=-\operatorname{div}(\nabla_{\xi}B(H(\nabla u)))=|u|^{\alpha-1}u$ in $\mathbb R^n$, and established corresponding Liouville theorems for stable solutions. Similar results for solutions that are stable outside a compact set have also been obtained in these works.

In the following, we consider stable solutions and solutions that are stable outside a compact set to the equation \eqref{eq:anisotropic} with $p\geqslant 2$. The linearized operator $\mathcal{L}_u$ of \eqref{eq:anisotropic} with respect to $u$ is defined by 
\begin{equation}
    \left<\mathcal{L}_u \varphi,\psi\right>_{L^2}=\int_{\mathbb R^n}\frac{1}{p}\nabla^2_{\xi}H^p(\nabla u):\nabla\varphi\otimes\nabla\psi-\int_{\mathbb{R}}f'(u)\varphi\psi,\quad\forall\varphi,\psi\in C_c^1(\mathbb R^n),
\end{equation}
where $\mathcal{L}_u\varphi:=-\operatorname{div}\left[\frac 1 p\nabla_\xi^2H^p(\nabla u)\cdot\nabla\varphi\right]-f'(u)\varphi$.

\begin{definition}
    Let $u\in W^{1,p}_{\mathrm{loc}}(\mathbb R^n)\cap L_{\mathrm{loc}}^\infty(\mathbb R^n)$ be a solution to \eqref{eq:anisotropic}. We say that $u$
    \begin{itemize}
        \item is stable if $\left<\mathcal{L}_u \varphi,\varphi\right>_{L^2}\geqslant 0$ for all $\varphi\in C_c^1(\mathbb R^n)$;
        \item has finite Morse index if $\max\limits_{\mathcal{M}\subset C_c^1(\mathbb R^n)}\operatorname{dim}\{\mathcal{M} \mid\left<\mathcal{L}_u \varphi,\varphi\right>_{L^2}<0,~\forall\varphi\in\mathcal{M}\}<\infty$;
        \item is stable outside a compact set $\mathcal{K}$ if $\left<\mathcal{L}_u \varphi,\varphi\right>_{L^2}\geqslant 0$ for all $\varphi\in C_c^1(\mathcal{K}^c)$.
    \end{itemize}
\end{definition}

\begin{remark}\label{rek:finite-outside}
    It is well known that if $u$ has a finite Morse index, then $u$ is stable outside a compact set.
\end{remark}

\begin{theorem}\label{thm:stable}
    Let $p\geqslant 2$, and let $f\in C^1(\mathbb{R})$ satisfy $uf(u)>0$ for all $u\ne0$. Suppose there exist constants $\alpha_0$ and $\beta_0$ such that one of the following holds:
    
    (1) $p<n$ and $p-1<\alpha_0\leqslant\beta_0< h(\alpha_0)$, where
    $$h(t)=\frac{n(p-1)}{n-p}+\frac{p}{n-p}\cdot\frac{2t-(p-1)+2\sqrt{t[t-(p-1)]}}{p-1},$$
    
    (2) $p\geqslant n$ and $p-1<\alpha_0\leqslant\beta_0<\infty$.
    
    \noindent Assume further that $f$ satisfies $u^2f'(u)\geqslant\alpha_0 uf(u)$ for all $u\in\mathbb{R}$, and there exist $c>0$ and $M>0$ such that $uf(u)\geqslant c|u|^{\beta_0+1}$ for all $|u|<M$. Then every stable solution to \eqref{eq:anisotropic} is trivial, that is $u\equiv 0$.
\end{theorem}

\begin{remark}
    A direct calculation shows that $h(t)>t$ for all $t>p-1$, which implies 
    $$
    \{(\alpha_0,\beta_0)\mid p-1<\alpha_0\leqslant\beta_0 <h(\alpha_0)\}\ne\varnothing.
    $$
\end{remark}

When $f(u)=|u|^{\alpha_0-1}u$, Theorem \ref{thm:stable} reduces to the Joseph-Lundgren exponent result.

\begin{corollary}[Corollary 1.4 in \cite{Farina2024-B-anisotropic-stable}]\label{cor:stable-p-Laplacian}
    For $1\leqslant p-1<\alpha_0<p_\mathrm{JL}$, every stable solution to the equation $-\Delta_p^H u=|u|^{\alpha_0-1}u$ in $\mathbb R^n$ is trivial, where the Joseph-Lundgren exponent is given by 
    $$p_\mathrm{JL}=\begin{cases} 
        \infty,&n\leqslant\frac{p(p+3)}{p-1}, \\
        \frac{[(p-1)n-p]^2+p^2(p-2)-p^2(p-1)n+2p^2\sqrt{(p-1)(n-1)}}{(n-p)[(p-1)n-p(p+3)]},&n>\frac{p(p+3)}{p-1}.
    \end{cases}$$
\end{corollary}

Using Pohozaev identity in the anisotropic setting (Theorem 1.3 in \cite{Montoro2023-Pohozaev}), we obtain the following Liouville result.

\begin{theorem}\label{thm:stable-outside}
    Let $p\geqslant 2$, and let $f\in C^1(\mathbb{R})$ satisfy $uf(u)>0$ for all $u\ne 0$. Suppose there exist constants $\alpha_0$, $\beta_0$, $c>0$, and $M>0$ such that one of the following conditions holds:
    
    (1) $p<n$, $p-1<\alpha_0\leqslant\beta_0<p^*-1$ and $f$ satisfies
    \begin{equation}\label{eq:f-outside-p<n}
        \alpha_0 uf(u)\leqslant u^2f'(u)\leqslant \beta_0 uf(u),\quad\forall u\in\mathbb{R}.
    \end{equation}
    
    (2) $p\geqslant n$, $p-1<\alpha_0\leqslant\beta_0<\infty$ and $f$ satisfies
    \begin{equation}\label{eq:f-outside-p>=n}
        u^2 f'(u)\geqslant\alpha_0 uf(u),\quad\forall u\in\mathbb{R},\quad\text{and}\quad uf(u)\geqslant c|u|^{\beta_0+1},\quad\forall |u|<M.
    \end{equation}
    
    \noindent Then every solution to \eqref{eq:anisotropic} that is stable outside a compact set $\mathcal{K}$ is trivial.
\end{theorem}

\begin{remark}
    The assumptions \eqref{eq:f-outside-p<n} and \eqref{eq:f-outside-p>=n} on the semilinear term $f$ are nontrivial. For $\alpha_0<\beta_0$, the following functions satisfy these conditions:
    $$f(u)=|u|^{\alpha_0-1}u+|u|^{\beta_0-1}u \quad \text{and} \quad f(u)= \frac{|u|^{\gamma-1}u}{\delta(1+|u|^{\tau})},$$
    where the parameters $\delta,\tau,\gamma$ satisfy $\alpha_0<\alpha_0+\tau\leqslant\gamma\leqslant\beta_0$ and $0<\delta\leqslant \frac{M^{\gamma-\beta_0}}{c(1+M^\tau)}$.
\end{remark}

As a consequence of Remark \ref{rek:finite-outside} and Theorem \ref{thm:stable-outside}, we obtain a Liouville result for finite Morse index solutions.

\begin{corollary}
    Under the assumptions of Theorem \ref{thm:stable-outside}, every finite Morse index solution to \eqref{eq:anisotropic} is trivial.
\end{corollary}

This paper is organized as follows. In Section \ref{sec:preliminaries}, we introduce preliminary results and establish key lemmas for subsequent proofs. In Section \ref{sec:serrin-technique}, we prove Theorems \ref{thm:Serrin-1} and \ref{thm:Serrin-2} by using Serrin's technique. In Section \ref{sec:subcritical}, we first consider the subcritical case of equation \eqref{eq:anisotropic} and prove the corresponding Liouville theorems (Theorems \ref{thm:subcritical} and \ref{cor:subcritical}). As an application, we then provide a new proof of Theorem \ref{thm:critical} and derive asymptotic estimates for local solutions near singular points (Theorem \ref{thm:asymptotic-1}). Finally, in Section \ref{sec:stable}, we establish Liouville results for both stable solutions and solutions stable outside a compact set (Theorems \ref{thm:stable} and \ref{thm:stable-outside}) by combining Serrin's technique with Pohozaev identity.

Throughout this paper, we assume $R>1$. The constant $\varepsilon$ denotes a sufficiently small positive constant, while $C$ denotes a positive constant independent of $R$. Let $\eta$ be a smooth cutoff function supported in $B_{2R}$ satisfying
$$0\leqslant\eta\leqslant 1, \quad\eta\equiv 1,\quad\text{in } B_R,\quad\text{and}\quad |\nabla^l\eta|\leqslant CR^{-l}\quad\text{for }l=1,2.$$
Unless otherwise specified, all $p$-(super, sub)harmonic functions considered here are anisotropic, all solutions are understood in the weak sense, all integrals are concentrated on $\mathbb R^n$, and we use the summation convention for repeated indices from $1$ to $n$.
\section{Preliminaries}\label{sec:preliminaries}

In this section, $\Omega \subset \mathbb R^n$ is a domain. We recall some properties of a uniformly convex, positively homogeneous function $H$.

\begin{lemma}[Lemma 3.1 in \cite{Sanali2009-Overdetermined}, Proposition 1.3 in \cite{Xia2012-Thesis}]
    For any $x,\xi\ne 0$, $l\in\mathbb R$,
    $$\nabla_\xi H^l\cdot\xi=lH^l,~\nabla H_0^l\cdot x=lH_0^l,~\nabla_\xi^2 H^l\cdot\xi=(l-1)\nabla_\xi H^l,~\nabla^2 H_0^l\cdot x=(l-1)\nabla H_0^l,$$
    $$H(\nabla H_0)=1,~H_0(\nabla_{\xi} H)=1,~H_0\nabla_{\xi} H(\nabla H_0)=x,~H\nabla H_0(\nabla_{\xi}H)=\xi.$$
\end{lemma}

\begin{lemma}[Proposition 3.1 in \cite{Farina2016-Monotonicity}]\label{lem:elliptic}
    There exist $\lambda,\Lambda>0$ depending on $n,p,H$, such that
    \begin{equation*}
        \lambda|\xi|^{p-2}|\zeta|^2\leqslant\frac{1}{p}\nabla^2_{\xi_i\xi_j}H^p(\xi)\zeta_i\zeta_j\leqslant\Lambda|\xi|^{p-2}|\zeta|^2,\quad\forall \zeta\in\mathbb R^n,~\xi\in\mathbb R^n\backslash\{0\}.  
    \end{equation*}
\end{lemma}

Set $Z=\{x\in \mathbb R^n\mid\nabla u(x)=0\}$. The following lemma provides regularity results.

\begin{lemma}[Theorems 1.1, 1.2, 1.4 and Proposition 1.6  in \cite{Farina2023-Interior-regularity}]\label{lem:regularity}
    Let $f\in L^r(\Omega)$, $r>n$. Suppose $u\in W^{1,p}_{\mathrm{loc}}(\Omega)\cap L^{\infty}_{\mathrm{loc}}(\Omega)$ is a solution to $-\Delta_p^H u=f$ in $\Omega$, then $u\in H_{\mathrm{loc}}^2(\Omega\backslash Z)\cap C^{1,q}_{\mathrm{loc}}(\Omega)$, $a(\nabla u)\in H^1_{\mathrm{loc}}(\Omega)$, and $f=0$, a.e. $x\in Z$, where $q\in (0,1)$ depends on $n,p,r,H$.
\end{lemma}

The positive semi-definite for invariant tensors plays a crucial role.

\begin{lemma}\label{lem:positivity}
    For the tensor $E_{ij}$ defined in \eqref{invariant} and constants $\lambda,\Lambda$ from Lemma \ref{lem:elliptic},
    \begin{equation}\label{eq 4.9}
        \frac{\Lambda}{\lambda} E_{ij}E_{ji}\geqslant\sum_{i,j=1}^n E_{ij}^2\geqslant 0.
    \end{equation}
\end{lemma}

\begin{proof}
    Denote $E=\{E_{ij}\}$. By using the $(p-1)$-homogeneity property of $\nabla_{\xi} H^p(\nabla u)$, $E=AB-\frac{1}{n} \operatorname{tr}(AB)I_n$, where $A=\frac{1}{p}\nabla^2_{\xi} H^p(\nabla u)$ and $B=\nabla^2 u-\frac{\nabla u \otimes \nabla u}{(p-1)\alpha(u)}$ are symmetric matrices.
    
    Choose $P\in O(n)$ such that $\tilde{A}:=P^{T}AP=\operatorname{diag}(a_1,\cdots,a_n)|\nabla u|^{p-2}$. Define $F=P^TEP$, then 
    $$F=\tilde{A}\tilde{B}-\frac{1}{n} \operatorname{tr} (\tilde{A}\tilde{B}) I_{n}$$
    with $\tilde{B}= P^{T}BP:=\{b_{ij}\}$. By Lemma \ref{lem:elliptic}, $\lambda\leqslant a_i\leqslant\Lambda$, $\forall 1\leqslant i \leqslant n$, thus for fixed $i\ne j$,
    \begin{equation*}
        F_{ij}^{2}=a_i^2 b_{ij}^2|\nabla u|^{2p-4} \leqslant \frac{\Lambda}{\lambda} a_i a_j b_{ij}^2 |\nabla u|^{2p-4} =\frac{\Lambda}{\lambda} F_{ij}F_{ji}.
    \end{equation*}
    Note that $\sum\limits_{i,j=1}^n E_{ij}^2 =\operatorname{tr}EE^T=\operatorname{tr}FF^T= \sum\limits_{i,j=1}^n F_{ij}^2$, $E_{ij}E_{ji} =\operatorname{tr}E^2=\operatorname{tr}F^2=F_{ij}F_{ji}$, hence 
    $$\sum\limits_{i,j=1}^n E_{ij}^2=\sum\limits_{i,j=1}^n F_{ij}^2\leqslant\sum_{i=1}^nF_{ii}^2+\frac{\Lambda}{\lambda}\sum_{i\ne j} F_{ij}F_{ji}\leqslant\frac{\Lambda}{\lambda}F_{ij}F_{ji}=\frac{\Lambda}{\lambda}E_{ij}E_{ji}.$$
\end{proof}

The following strong maximum principle is derived from Harnack inequality in Theorem 1.2 of \cite{Trudinger1967}. This result shows that, for anisotropic equations \eqref{eq:anisotropic} with $f\geqslant 0$, the study of nonnegative supersolutions can be reduced to that of positive supersolutions.

\begin{lemma}[Strong maximum principle]\label{lem:strong}
    Suppose $u\in W^{1,p}_{\mathrm{loc}}(\Omega) \cap C(\Omega)$ is a nonnegative $p$-superharmonic function, i.e., $-\Delta_p^H u\geqslant 0$ in $\Omega$ in weak sense. Then either $u\equiv 0$ or $u>0$ in $\Omega$.
\end{lemma}

We provide a proof of the weak maximum principle (comparison principle) for the anisotropic $p$-Laplacian, although it's well-known.

\begin{lemma}[Comparison principle]\label{lem:comparison}
    Assume $u,v \in W^{1,p}_{\mathrm{loc}}(\mathbb R^n)\cap L_{\mathrm{loc}}^\infty(\mathbb R^n)$ satisfy $-\Delta_p^H u\leqslant -\Delta_p^H v$ weakly in $\Omega$, and $u\leqslant v$ weakly on $\partial \Omega$, i.e. $D_\varepsilon:=\{u-v>\varepsilon\}\subset\subset\Omega$ for every $\varepsilon>0$. Then $u\leqslant v$ in $\Omega$.
\end{lemma}

\begin{proof}
    $\forall \varepsilon>0$, by testing $\min\{v(x)-u(x)+\varepsilon,0\}$ in $-\Delta_p^H u\leqslant -\Delta_p^H v$, we can deduce
    \begin{equation}\label{ineq:comparison}
        \int_{D_\varepsilon} [a(\nabla u)-a(\nabla v)] \cdot [\nabla u-\nabla v]\leqslant 0.
    \end{equation} 
    According to $(38)$ in \cite{Farina2023-Interior-regularity}, there exists a positive constant $\gamma_0$ depending only on $n,p,H$ such that  
    \begin{equation*}
        [a(\nabla u)-a(\nabla v)]\cdot [\nabla u-\nabla v]\geqslant \gamma_0
        \left\{\begin{aligned}
            & (1+|\nabla u|+|\nabla v|)^{p-2}|\nabla u-\nabla v|^2, && 1<p<2,\\
            & |\nabla u-\nabla v|^p, && p\geqslant 2.
        \end{aligned}\right.
    \end{equation*} 
    Hence \eqref{ineq:comparison} yields $\nabla u=\nabla v$ a.e. $D_\varepsilon$, which implies $u=v+\varepsilon$ in $ D_\varepsilon$, thus $u\leqslant v+\varepsilon$ in $\Omega$.
\end{proof}

As a consequence, the asymptotic behavior of positive $p$-superharmonic functions on exterior domains holds.

\begin{lemma}[Asymptotic behavior]\label{lem:asymptotic}
    Let $R_0>0$ and $B_{R_0}^c\subset\Omega$. Suppose $u \in W^{1,p}_{\mathrm{loc}}(\Omega)\cap C(\Omega)$ is $p$-superharmonic, then the following asymptotic estimates hold:
    
    (1) When $1<p<n$, $u(x)\geqslant C|x|^{-\frac{n-p}{p-1}}$ on $B_{2R_0}^c$ where $C>0$ depends on $n,p,u,H,R_0$.
    
    (2) When $p\geqslant n$, $\liminf\limits_{|x|\to \infty}u(x)>0$.
\end{lemma}
\begin{proof}
    (1) Set the fundamental solution $v(x)=\hat{H}_0^q(x)$, where $q=-\frac{n-p}{p-1}$. In fact, for $x\ne 0$,
    \begin{align*}
        \Delta_p^H v(x)=~&\operatorname{div}[a(q\hat{H}_0^{q-1}(x)\nabla\hat{H}_0(x))]=\operatorname{div}\{[-qH_0^{q-1}(-x)]^{p-1}a(\nabla H_0(-x)) \}\\
        =~&(-q)^{p-1}\operatorname{div}\{[H_0(-x)]^{(p-1)(q-1)-1}(-x)\}=(-q)^{p-1}(-n-pq+p+q)[H_0(-x)]^{pq-p-q}=0.
    \end{align*}

    Define $K_0=\min\limits_{|x|=2R_0}u(x) \cdot \min\limits_{|x|=2R_0}\hat H_0^{\frac{n-p}{p-1}}>0$ and $w=K_0 v$, then $w$ satisfies
    $$\begin{cases}
        -\Delta_p^H w\leqslant -\Delta_p^H u,&\text{on } B_{2R_0}^c,\\
        w=\min\limits_{|x|=2R}u(x) \leqslant u,&\text{on } \partial B_{2R_0},\\
        \limsup\limits_{|x|\to \infty} [w(x)-u(x)]\leqslant 0.
    \end{cases}$$
    It follows from Lemma \ref{lem:comparison} that $
    u(x)\geqslant w(x) \geqslant C |x|^{-\frac{n-p}{p-1}}$ for all $x\in B_{2R_0}^c$.

    (2) Let $\lambda_1=\min\limits_{|\xi|=1}H_0(\xi)$, $\lambda_2=\max\limits_{|\xi|=1}H_0(\xi)$ and $S>\frac{\lambda_2}{\lambda_1}R_0$. Define $h(x)=\log_{\frac{2S}{R_0}}\frac{\min\limits_{|y|=4S}\hat{H}_0(y)}{\hat{H}_0(x)}$, $x \in \Omega$, then
    \begin{align*}
        \Delta_p^H h(x)=~&\operatorname{div}\left[a\left(\left(H_0(-x)\log\frac{2S}{R_0}\right)^{-1}\nabla H_0(-x)\right)\right]\\
        =~&\left(\log\frac{2S}{R_0}\right)^{1-p}\operatorname{div}\{[H_0(-x)]^{-p}(-x)\}=\left(\log\frac{2S}{R_0}\right)^{1-p}\frac{p-n}{[H_0(-x)]^p} \geqslant0. 
    \end{align*}
    Let $\varepsilon_0=\min\limits_{|x|=2R_0}u(x)>0$. Note that $h\leqslant 1$ on $\partial B_{2R_0}$, and $h\leqslant 0$ on $\partial B_{4S}$. By applying Lemma \ref{lem:comparison} to $u$ and $\varepsilon_0 h$, it follows that $u \geqslant \varepsilon_0 h$ on $\overline{B_{4S}}\backslash B_{2R_0}$. In particular, when $|x|=2\sqrt{\frac{2\lambda_1 SR_0}{\lambda_2}}$, $h(x) \geqslant \log_{\frac{2S}{R_0}}\sqrt{\frac{2S \lambda_1}{R_0 \lambda_2}}$. By taking $S\to \infty$, we conclude that $\liminf\limits_{|x|\to \infty} u(x)\geqslant \frac{\varepsilon_0}{2}=\frac{1}{2}\min_{|x|=2R_0}u(x)>0$.
\end{proof}

The following corollary explains why our discussion of nonnegative supersolutions to $-\Delta_p^H u=f(u)\geqslant 0$ is concentrated on $1<p<n$.

\begin{corollary}\label{cor:p>=n}
    When $p\geqslant n$, any nonnegative superharmonic function is constant.
\end{corollary}
\begin{proof}
    For $p$-superharmonic function $u$, consider the nonnegative $p$-superharmonic function $w:=u-\inf\limits_{\mathbb R^n} u$. By Lemma \ref{lem:strong}, either $w \equiv 0$ or $w > 0$. In the latter case, we would have $\liminf\limits_{|x| \to \infty} w(x) > 0$, which contradicts the fact that $\inf\limits_{\mathbb R^n}w =0$. Therefore, we must have $w \equiv 0$, and $u$ is constant.
\end{proof}
\section{Serrin's technique: proofs of Theorems \ref{thm:Serrin-1} and \ref{thm:Serrin-2}}\label{sec:serrin-technique}

\begin{proof}[Proof of Theorem \ref{thm:Serrin-1}]
    Let $u$ be a positive supersolution to \eqref{eq:anisotropic}. Taking $\varphi=\beta(u)\eta^n$ in \eqref{weak-supsol}, where $\beta\in C^1(\mathbb{R}_+)$, $\beta>0$ and $\beta'<0$, we deduce
    \begin{equation}\label{ineq:Serrin-1}
        \int \beta(u)f(u)\eta^n-\int\beta'(u)H^p(\nabla u)\eta^n\leqslant n\int_{B_{2R}\backslash B_R}\beta(u)H^{p-1}(\nabla u)\nabla_\xi H(\nabla u)\cdot\eta^{n-1}\nabla\eta.
    \end{equation}
    By using $|\nabla_\xi H(\nabla u)|\leqslant C$ and Young's inequality with indices $\left(\frac{np}{n-p},~\frac{p}{p-1},n\right)$, we have
    \begin{align}
        \begin{split}\label{ineq:Serrin-2}
            &n\int_{B_{2R}\backslash B_R}\beta(u)H^{p-1}(\nabla u)\nabla_\xi H(\nabla u)\cdot\eta^{n-1}\nabla\eta\\
            \leqslant~&A\int_{B_{2R}\backslash B_R} \beta(u)f(u)\eta^n-A\int_{B_{2R}\backslash B_R}\beta'(u)H^p(\nabla u)\eta^n\\
            &+CA^{1-n}\int_{B_{2R}\backslash B_R}[f(u)]^{\frac{p-n}{p}}[\beta(u)]^{-n\left(\frac{n-p}{np}-1\right)}[-\beta'(u)]^{-\frac{p-1}{p}n}|\nabla\eta|^n,
        \end{split}
    \end{align}
    where $A>0$. Set $\tilde\beta(u)=[\beta(u)]^{-\frac{p}{(p-1)n}}$ and take $A=\frac{1}{2}$. Then \eqref{ineq:Serrin-1} becomes
    $$\int_{B_R} \beta(u)f(u)-\int_{B_R}\beta'(u)H^p(\nabla u)\leqslant C\int_{B_{2R}\backslash B_R}[f(u)]^{\frac{p-n}{p}}[\tilde\beta'(u)]^{-\frac{p-1}{p}n}|\nabla\eta|^n.$$
    
    Define $\displaystyle\tilde\beta(u)=\int_0^u[f(v)]^{-\frac{1}{p_*-1}}\mathrm dv$. Then $\beta\in C^1(\mathbb R_+)$, $\beta>0$, $\beta'<0$ and
    $$\int_{B_R} \beta(u)f(u)-\int_{B_R}\beta'(u)H^p(\nabla u)\leqslant C,$$
    which implies $\lim_{R\to\infty}\left(\int_{B_{2R}\backslash B_R} \beta(u)f(u)-\int_{B_{2R}\backslash B_R}\beta'(u)H^p(\nabla u)\right)=0$.
    
    For any $A>0$, by substituting \eqref{ineq:Serrin-2} into \eqref{ineq:Serrin-1} and letting $R \to \infty$, we conclude
    $$\int \beta(u)f(u)\eta^n-\int\beta'(u)H^p(\nabla u)\eta^n\leqslant CA^{1-n}.$$
    The conclusion follows by letting $A\to\infty$.
\end{proof}

Following a similar argument, we prove Theorem \ref{thm:Serrin-2}.

\begin{proof}[Proof of Theorem \ref{thm:Serrin-2}]
    Let $u$ be a positive subsolution to \eqref{eq:anisotropic}. Taking $\varphi=\beta(u)\eta^\gamma$ in \eqref{weak-supsol}, where $\beta\in C^1(\mathbb R_+)$, $\beta>0$, $\beta'>0$, and $\gamma>0$ is a constant to be determined, we deduce
    \begin{align*}
        &\int\beta(u)[-f(u)]\eta^\gamma+\int\beta'(u)H^p(\nabla u)\eta^\gamma\leqslant-\gamma\int\beta(u)H^{p-1}(\nabla u)\nabla_\xi H(\nabla u)\cdot\eta^{\gamma-1}\nabla\eta\\
        \leqslant~&\frac{1}{2}\int\beta(u)[-f(u)]\eta^\gamma+\frac{1}{2}\int \beta'(u)H^p(\nabla u)\eta^\gamma+C\int[-f(u)]^{-\frac{p}{a-p}}[\beta(u)]^{\frac{p(a-1)}{a-p}}[\beta'(u)]^{-\frac{(p-1)a}{a-p}}\eta^{\gamma-\frac{pa}{a-p}}|\nabla \eta|^{\frac{pa}{a-p}},
    \end{align*}
    where we use Young's inequality with indices $\left(a,~\frac{p}{p-1},~\frac{pa}{a-p}\right)$ for some $a>p$. By taking $\tilde\beta(u)=-[\beta(u)]^{\frac{p-a}{(p-1)a}}$, $\gamma=\frac{pa}{a-p}$ and $a=\frac{pq}{p-1}>p$, we obtain
    $$\int\beta(u)[-f(u)]\eta^\gamma+\int\beta'(u)H^p(\nabla u)\eta^\gamma\leqslant C\int \left[[-f(u)]^{\frac{1}{q}}\tilde\beta'(u)\right]^{-\frac{q(p-1)}{q-p+1}}|\nabla\eta|^{\frac{pq}{q-p+1}}.$$    

    Define $\tilde\beta(u)=\int_{+\infty}^{u}[-f(v)]^{-\frac{1}{q}}\mathrm{d}v$, then $\beta\in C^1(\mathbb R_+)$, $\beta>0$, $\beta'>0$, and
    $$\int\beta(u)[-f(u)]\eta^\gamma+\int\beta'(u)H^p(\nabla u)\eta^\gamma\leqslant CR^{n-\frac{pq}{q-p+1}}.$$ 
    The conclusion follows by letting $R\to\infty$.
\end{proof}
\section{Subcritical case: Liouville theorems and applications}\label{sec:subcritical}

In this section, we first study positive solutions to \eqref{eq:anisotropic} in the subcritical case under certain assumptions on $f$ and prove the corresponding Liouville results via the invariant tensor technique (Theorems \ref{thm:subcritical} and \ref{cor:subcritical}). Next, we provide a new proof of Theorem \ref{thm:critical}, the critical case. Finally, as an application of these Liouville results, we study the asymptotic behavior of nonnegative (local) solutions near singular points and prove Theorem \ref{thm:asymptotic-1}.

\subsection{Subcritical case: proofs of Theorems \ref{thm:subcritical} and \ref{cor:subcritical}}\label{subsec:proof-thm:subcritical}
Assume $1<p<n$, $u$ is a positive solution to \eqref{eq:anisotropic}. By Lemma \ref{lem:regularity} and bootstrapping regularity arguments, there exists $q\in (0,1)$ depending only on $n$, $p$, and $H$, such that
$u \in C^3(Z^c) \cap C_{\mathrm{loc}}^{1,q}(\mathbb R^n)$, $a(\nabla u) \in H^1_{\mathrm{loc}}(\mathbb R^n)$.

Let $\alpha:\mathbb{R}_+\to \mathbb{R}_+$ be a $C^1$ function. We consider the trace-free tensor
\begin{equation}\label{invariant}   
    E_{ij}=\begin{cases}
        \partial_j(a_i(\nabla u))-\frac{a_i(\nabla u)u_j}{\alpha(u)}-\frac{1}{n} \left(\Delta_p^H u-\frac{H^p(\nabla u)}{\alpha(u)}\right)\delta_{ij},&\quad\text{for } |\nabla u|\ne 0,\\
        0,&\quad\text{for }|\nabla u|=0,
    \end{cases}
\end{equation}
where $a_i(\nabla u)=\frac 1 p\nabla _{\xi_i}H^p(\nabla u)$ satisfies $a_i(\nabla u)u_i=H^p(\nabla u)$. Now, we prove the key differential identity. 

\begin{lemma}\label{lem:subcritical-1}
    The following differential identity holds:
    \begin{equation}\label{id}
        \partial_i [\beta(u)E_{ij}a_j(\nabla u)]=\beta(u)\left[E_{ij}E_{ji}+A\frac{H^{2p}(\nabla u)}{u^2}-\frac{n-1}{n}\left(f'(u)-\frac{\alpha_0}{u}f(u)\right)H^{p}(\nabla u)\right].
    \end{equation}
    where $A=\frac{n-1}{n}\frac{p_*-1}{p^*-1}\alpha_0\left(1-\frac{\alpha_0}{p^*-1}\right)$, $\alpha(u)=\frac{p^*-1}{(p_*-1)\alpha_0}u$ and $\beta(u)=u^{-\frac{p_*-2}{p^*-1}\alpha_0}$.
\end{lemma}

\begin{proof}
    All subsequent analysis will be confined to $Z^c$. A direct computation yields
    $$E_{ij}u_i=(p-1)a_k(\nabla u)u_{kj}-\frac{n-1}{n}\frac{H^p(\nabla u)}{\alpha(u)}u_j+\frac{1}{n}f(u)u_j.$$
    \begin{align*}
        \partial_i E_{ij}=~&\frac{n-1}{n}\partial_j(\Delta_p^H u)-\frac{\Delta_p^H u}{\alpha(u)}u_j -\frac{n-p}{n}\frac{a_i(\nabla u)}{\alpha(u)}u_{ij}+\frac{n-1}{n}\frac{\alpha'(u)}{\alpha^2(u)}H^p(\nabla u) u_j\\
        =~&-\frac{1}{p_*-1}\frac{E_{ij}u_i}{\alpha(u)}-\frac{n-1}{n}\left(f'(u)-\frac{p^*-1}{p_*-1}\frac{f(u)}{\alpha(u)}\right)u_j+\frac{n-1}{n}\frac{1}{\alpha^2(u)}\left(\alpha'(u)-\frac{1}{p_*-1}\right)H^p(\nabla u)u_j,
    \end{align*}
    Let $\beta:\mathbb{R}_+\to\mathbb{R}_+$ be a $C^1$ function, then
    \begin{align*}
        &[\beta(u)]^{-1}\partial_i [\beta(u)E_{ij}a_j(\nabla u)]\\
        =~&E_{ij}E_{ji}+\left(\frac{\beta'(u)}{\beta(u)}+\frac{p_*-2}{p_*-1}\frac{1}{\alpha(u)}\right)E_{ij}a_j(\nabla u)u_i\\
        &-\frac{n-1}{n}\left(f'(u)-\frac{p^*-1}{p_*-1}\frac{f(u)}{\alpha(u)}\right)H^{p}(\nabla u)+\frac{n-1}{n}\frac{1}{\alpha^2(u)}\left(\alpha'(u)-\frac{1}{p_*-1}\right)H^{2p}(\nabla u).
    \end{align*}
    By substituting $\alpha(u)=\frac{p^*-1}{(p_*-1)\alpha_0}u$ and $\beta(u)=u^{-\frac{p_*-2}{p^*-1}\alpha_0}$ into the above identity, we obtain \eqref{id}.
\end{proof}

\begin{remark}
    The invariant tensor technique is reflected in taking $\alpha(u)=\frac{p^*-1}{(p_*-1)\alpha_0}u$ in $E_{ij}$. We recommend readers to see \cite{Ma-Wu-2024} for more details.
\end{remark}

\begin{lemma}\label{lem:subcritical-2}
    Set $g(\alpha_0)=-\frac{p_*-2}{p^*-1}\alpha_0+2p-2$. In condition of Theorem \ref{thm:subcritical} or \ref{cor:subcritical}, for any $\gamma>0$,
    \begin{equation}\label{ineq:integral}
        \int u^{-\frac{p_*-2}{p^*-1}\alpha_0}f^2(u)\eta^\gamma\leqslant C\int u^{g(\alpha_0)}\eta^{\gamma-2p}|\nabla\eta|^{2p}.
    \end{equation} 
\end{lemma}
\begin{proof}
    By the regularity of $u$, both sides in \eqref{id} are locally integrable. Hence by testing $\eta^\gamma$ in \eqref{id} and applying the condition $f'(u)\leqslant\alpha_0\frac{f(u)}{u}$, we derive
    $$\int\beta(u)\left[E_{ij}E_{ji}+A\frac{H^{2p}(\nabla u)}{u^2}\right]\eta^\gamma\leqslant-\gamma\int\beta(u)H^{p-1}(\nabla u) \eta^{\gamma-1} E_{ij}\nabla_{\xi_j}H(\nabla u)\eta_i.$$
    By using Lemma \ref{lem:positivity} and Young's inequality, we can deduce
    \begin{align*}
        &-\gamma \int \beta(u)  H^{p-1}(\nabla u) \eta^{\gamma-1} E_{ij}\nabla_{\xi_j}H(\nabla u )  \eta_i\\
        \leqslant~&\frac 1 2\int \beta(u)E_{ij}E_{ji}\eta^\gamma +\frac A 2\int \beta(u) \frac{H^{2p}(\nabla u)}{u^2} \eta^\gamma+C  \int \beta(u) u^{2p-2} \eta^{\gamma-2p}|\nabla \eta|^{2p}.
    \end{align*}
    Hence
    \begin{equation}\label{ineq:subcritical-1}
        \int \beta(u)\left[E_{ij}E_{ji}+\frac{H^{2p}(\nabla u)}{u^2} \right]\eta^\gamma \leqslant  C \int \beta(u)  u^{2p-2} \eta^{\gamma-2p}|\nabla \eta|^{2p}.
    \end{equation}

    Taking $\varphi=\beta(u)f(u)\eta^\gamma$ in \eqref{weak-sol}, we obtain
    \begin{align*}
        \int\beta(u)f^2(u)\eta^\gamma 
        =~&\int\beta(u)f'(u)H^p(\nabla u)\eta^\gamma+\int\beta'(u)H^p(\nabla u)f(u)\eta^\gamma \\
        &+\gamma\int\beta(u)f(u)H^{p-1}(\nabla u)\nabla_{\xi_i}H(\nabla u)\eta^{\gamma-1}\eta_i.
    \end{align*}
    By combining $f'(u)\leqslant\alpha_0\frac{f(u)}{u}$, $|\beta'(u)|\leqslant C\frac{\beta(u)}{u}$, and Young's inequality, we conclude 
    \begin{align*}
        \int \beta(u) f^2(u)\eta^\gamma 
        \leqslant ~& C\int\frac{\beta(u)}{u}|f(u)|H^p(\nabla u) \eta^\gamma+C\int\beta(u)|f(u)|H^{p-1}(\nabla u) \eta^{\gamma-1}|\nabla\eta| \\
        \leqslant ~& C\int \beta(u)\frac{H^{2p}(\nabla u)}{u^2} \eta^\gamma +\frac{1}{2}\int \beta(u)f^2(u)\eta^\gamma +C \int \beta(u) u^{2p-2} \eta^{\gamma-2p}|\nabla \eta|^{2p}. 
    \end{align*}
    By substituting \eqref{ineq:subcritical-1} and $\beta(u)=u^{-\frac{p_*-2}{p^*-1}\alpha_0}$ into this inequality, we obtain \eqref{ineq:integral}.
\end{proof}

We use the integral inequality \eqref{ineq:integral} to prove Theorem \ref{thm:subcritical}.
\begin{proof}[Proof of Theorem \ref{thm:subcritical}]
    Define $\mathcal{X}_1=\{u\leqslant M\}$ and $\mathcal{X}_2=\{u> M\}$. By the assumptions of $f$, $f(u)\geqslant\frac{f(M)}{M^{\alpha_0}}u^{\alpha_0}$ on $\mathcal{X}_1$, and $f(u)\geqslant cu^{\beta_0}$ in $\mathcal{X}_2$. Thus \eqref{ineq:integral} becomes
    \begin{equation}\label{ineq:subcritical-2}
        \int_{\mathcal{X}_1} u^{\left(2-\frac{p_*-2}{p^*-1}\right)\alpha_0} \eta^\gamma+\int_{\mathcal{X}_2} u^{2\beta_0 -\frac{p_*-2}{p^*-1}\alpha_0}\eta^\gamma \leqslant C\int u^{g(\alpha_0)} \eta^{\gamma-2p} |\nabla \eta|^{2p}.
    \end{equation}

    In the following, we divide into two cases. It is easy to check that $g(p-1)=\frac{(p-1)p(n+1)}{np-n+p}>0$.

    \noindent \textbf{Case 1.} When $g(\alpha_0)> 0$, we apply Young's inequality and take $\gamma=\left(2\beta_0-\frac{p_*-2}{p^*-1}\alpha_0\right)\frac{p}{\beta_0-p+1}$,  so that the right-hand side of \eqref{ineq:subcritical-2} becomes
    \begin{align*}
        C\int u^{g(\alpha_0)} \eta^{\gamma-2p} |\nabla \eta|^{2p}
        \leqslant~& \frac{1}{2}\int_{\mathcal{X}_1} u^{\left(2-\frac{p_*-2}{p^*-1}\right)\alpha_0}\eta^\gamma+C R^{n-\left(2-\frac{p_*-2}{p^*-1}\right)\frac{p\alpha_0}{\alpha_0-p+1}} \\
        &+ \frac{1}{2} \int_{\mathcal{X}_2} u^{2\beta_0 -\frac{p_*-2}{p^*-1}\alpha_0}\eta^\gamma+CR^{n-\left(2\beta_0-\frac{p_*-2}{p^*-1}\alpha_0\right)\frac{p}{\beta_0-p+1}}.
    \end{align*}
    Substituting it into \eqref{ineq:subcritical-2} and using $n-\left(2\beta_0-\frac{p_*-2}{p^*-1}\alpha_0\right)\frac{p}{\beta_0-p+1}\leqslant n-\left(2-\frac{p_*-2}{p^*-1}\right)\frac{p\alpha_0}{\alpha_0-p+1}$, 
    we obtain
    $$\int_{\mathcal{X}_1} u^{\left(2-\frac{p_*-2}{p^*-1}\right)\alpha_0} \eta^\gamma+\int_{\mathcal{X}_2} u^{2\beta_0 -\frac{p_*-2}{p^*-1}\alpha_0}\eta^\gamma\leqslant CR^{n-\left(2-\frac{p_*-2}{p^*-1}\right)\frac{p\alpha_0}{\alpha_0-p+1}},$$
    where
    \begin{align*}
        &n-\left(2-\frac{p_*-2}{p^*-1}\right)\frac{p\alpha_0}{\alpha_0-p+1}=n-\frac{(n+1)p^2\alpha_0}{(n-p)(p^*-1)(\alpha_0-p+1)}\\
        <~&n-\frac{np^2\alpha_0}{(n-p)(p^*-1)(\alpha_0-p+1)}=\frac{n(p-1)(\alpha_0-p^*+1)}{(p^*-1)(\alpha_0-p+1)}<0.
    \end{align*}

    \noindent \textbf{Case 2.} When $g(\alpha_0) \leqslant 0$, since $g$ is linear in $\alpha_0$, we conclude that $g(p^*-1)<0$. It follows that $p_*>2p$, that is $p>\frac{n+1}{2}$. By using Lemma \ref{lem:asymptotic} and taking $\gamma = 2p$, we can reduce \eqref{ineq:integral} to
    $$\int u^{-\frac{p_*-2}{p^*-1}\alpha_0}f^2(u)\eta^\gamma \leqslant CR^{n-2p-\frac{n-p}{p-1}g(\alpha_0)}<CR^{n-2p-\frac{n-p}{p-1}g(p^*-1)}=CR^{-\frac{n-p}{p-1}}.$$
    
     In both cases, the conclusion follows by letting $R\to\infty$.
\end{proof}

When $f$ changes sign, the asymptotic behavior cannot be applied to Case 2. However, the argument for $\mathcal{X}_1$ in Case 1 remains valid and suffices to prove Theorem \ref{cor:subcritical}.

\begin{proof}[Proof of Theorem \ref{cor:subcritical}]
    Set $N=\sup |u|<\infty$, then $f(u)\geqslant \frac{f(N)}{N^{\alpha_0}}u^{\alpha_0}$ in $\mathbb R^n=\{u\leqslant N\}$.
    
    When $g(\alpha_0)>0$, by following the argument in Case 1 in the proof of Theorem \ref{thm:subcritical}, we derive 
    $$\int u^{\left(2-\frac{p_*-2}{p^*-1}\right)\alpha_0}\eta^\gamma\leqslant \frac{1}{2}\int u^{(2-\frac{p_*-2}{p^*-1})\alpha_0}\eta^\gamma+C R^{\frac{n(p-1)(\alpha_0-p^*+1)}{(p^*-1)(\alpha_0-p+1)}},$$
    where $\gamma=\left(2-\frac{p_*-2}{p^*-1}\right)\frac{p\alpha_0}{\alpha_0-p+1}$.

    When $g(\alpha_0)=0$, we have $p>\frac{n+1}{2}$. By taking $\gamma=2p$ in \eqref{ineq:integral}, we obtain
    $$\int u^{-\frac{p_*-2}{p^*-1}\alpha_0}f^2(u)\eta^\gamma\leqslant CR^{n-2p}.$$
    In both cases, the conclusion follows by taking $R \to \infty$.
\end{proof}

\subsection{Critical case: proof of Theorem \ref{thm:critical}}\label{subsec:critical}

In this subsection, we use the key differential identity \eqref{id}, where $\alpha_0=p^*-1$, to prove Theorem \ref{thm:critical}.

\begin{proof}[Proof of Theorem \ref{thm:critical}]
    The proof proceeds in four steps.

    \textbf{Step 1. Regularity.} By Lemma \ref{lem:regularity}, $u\in C_{\mathrm{loc}}^{1,q}(\mathbb R^n)$ and $a(\nabla u)\in H^1_{\operatorname{loc}}(\mathbb R^n)$, thus we do not need to use a regularization argument.

    \textbf{Step 2. Asymptotic behavior.} By Proposition 2.3 in \cite{Figalli2020}, there exist two positive constants $C_0$ and $C_1$ such that 
	\begin{equation}\label{ineq:asymptotic-critical}
        \frac{C_0}{1+|x|^{\frac{n-p}{p-1}}}\leqslant u(x)\leqslant \frac{C_1}{1+|x|^{\frac{n-p}{p-1}}}\quad\text{and}\quad|\nabla u(x)|\leqslant \frac{C_1}{1+|x|^{\frac{n-1}{p-1}}},\quad \forall x\in\Sigma.
    \end{equation}

    \textbf{Step 3. Differential identity.} Similar to the subcritical case (see \eqref{id}), we have the following differential identity:
    \begin{align}\label{id-critical}
        \partial_i[u^{2-p_*}E_{ij}a_j(\nabla u)]=u^{2-p_*}E_{ij}E_{ji},
    \end{align}
    where $E_{ij}=\partial_{j}(a_i(\nabla u))-(p_{*}-1)\frac{a_i(\nabla u)u_{j}}{u}-\frac{1}{n}\left(\Delta^H_{p}u-(p_{*}-1)\frac{H^p(\nabla u)}{u}\right)\delta_{ij}$.

    By testing $\eta^2$ in \eqref{id-critical} and applying Young's inequality, we obtain
    \begin{equation*}
        \int u^{2-p_{*}}E_{ij}E_{ji}\eta^2\leqslant C\int u^{2-p_{*}}H^{2p-2}(\nabla u)|\nabla\eta|^{2}\leqslant CR^{-\frac{n-p}{p-1}}.
    \end{equation*}
    Here, we use Lemma \ref{lem:positivity} and the asymptotic behavior \eqref{ineq:asymptotic-critical}. Therefore $E_{ij}\equiv 0$, $\forall 1\leqslant i,j \leqslant n$.

    \textbf{Step 4. Classification: $u\equiv u_{\lambda,x_0}^H $.} It follows form $E_{ij}\equiv 0$ that $\partial_{j}[u^{1-p_{*}}H^{p-1}(\nabla u)\nabla_{\xi_{j}}H(\nabla u)]=0$ and 
    $$\partial_i[u^{1-p_{*}}H^{p-1}(\nabla u)\nabla_{\xi_i}H(\nabla u)]=\partial_j[u^{1-p_{*}}H^{p-1}(\nabla u)\nabla_{\xi_j}H(\nabla u)]$$
    for any $i\ne j$. 
    Then there exist $x_{0}\in\mathbb R^{n}$ and a constant $a_0\ne 0$ such that
    \begin{equation}\label{ineq:critical-1}
        u^{1-p_{*}}H^{p-1}(\nabla u)\nabla_{\xi}H(\nabla u)=a_0(x-x_{0}).
    \end{equation}
    By applying $H_{0}$ and $\nabla H_{0}$ respectively to both sides of \eqref{ineq:critical-1}, we obtain
    $$u^{1-p_{*}}H^{p-1}(\nabla u)=a_0 H_{0}(x-x_{0}),~\frac{\nabla u}{H(\nabla u)}=\nabla H_{0}(x-x_{0}).$$
    By eliminating $ H(\nabla u)$, we obtain the explicit solution $u\equiv u_{\lambda,x_0}^H$.
\end{proof}

\subsection{Application: singularity and decay estimates via Liouville theorems}\label{subsec:application}

In this subsection, we prove Theorem \ref{thm:asymptotic-1}. Although the proof is similar to that of Theorem 3.1 in \cite{Souplet2007}, we provide full details here for the sake of completeness.

\begin{proof}[Proof of Theorem \ref{thm:asymptotic-1}]
    Let $\delta=\frac{p}{\alpha_0-p+1}$. If estimate \eqref{eq:asymptotic-estimate 1} fails, then there exist $\Omega_k$, $u_k$, and $y_k \in \Omega_k$ such that $-\Delta_p^H u_k=f(u_k)$ in $\Omega_k$ and
    \begin{equation}\label{ineq:subcritical-3}
        M_k(y_k):= u(y_k)^{\frac{1}{\delta}}+|\nabla u_k(y_k)|^{\frac{1}{\delta+1}}>2k\left(1+d^{-1}(y_k)\right)>2k d^{-1}(y_k).
    \end{equation}
    By the doubling lemma (Lemma 5.1 and Remark 5.2 (b) in \cite{Souplet2007}), there exist points $x_k\in\Omega_k$ such that
    \begin{equation}\label{ineq:subcritical-4}
        M_k(x_k)\geqslant M_k(y_k),\quad M_k(x_k)>2k d^{-1}(x_k),\quad\text{and }M_k(z)\leqslant 2M_k(x_k)\text{ for } |z-x_k|\leqslant k M_k^{-1}(x_k).
    \end{equation}
    Define $\lambda_k=M_k^{-1}(x_k)$ and $v_k(y)=\lambda_k^{\delta}u_k(x_k+\lambda_k y)$ for $y\in \overline{B_k}$. By \eqref{ineq:subcritical-3}, we have $\lambda_k \to 0$. Moreover
    \begin{equation}\label{ineq:subcritical-5}
        v_k(0)^{\frac{1}{\delta}}+|\nabla v(0)|^{\frac{1}{\delta+1}}=1 \quad \text{and} \quad v_k(y)^{\frac{1}{\delta}}+|\nabla v(y)|^{\frac{1}{\delta+1}}\leqslant 2,\quad\forall y\in\overline{B_k}.
    \end{equation}
    
    It is easy to check that each $v_k$ satisfies
    \begin{equation}\label{ineq:subcritical-6}
        -\Delta_p^H v_k(y)=f_k(v_k(y)):=\lambda_k^{(\delta+1)(p-1)+1}f(\lambda_k^{-\delta}v_k(y)),\quad\forall y\in\overline{B_k},
    \end{equation}
    From the condition $-C\leqslant f(s)\leqslant C(1+s^{\alpha_0})$, we derive uniform bounds
    $$
    -C\lambda_k^{(\delta+1)(p-1)+1}\leqslant f_k(v_k(y))\leqslant C', \quad\forall y\in\overline{B_k}, \quad\forall k\geqslant 1.
    $$
    By Lemma \ref{lem:regularity}, $v_k \in C^{1,q}_{\mathrm{loc}}(B_k)$ for some $q\in (0,1)$. Thus, a subsequence $v_{k_j}$ converges in $C^{1,q'}_{\mathrm{loc}}(\mathbb R^n)$ to a function $v$ for some $q'\in (0,1)$, where $v \geqslant 0$, $-\Delta_p^H v \geqslant 0$, and $v(0)^{\frac{1}{\delta}}+|\nabla v(0)|^{\frac{1}{\delta+1}}=1$. The strong maximum principle (Lemma \ref{lem:strong}) gives $v>0$ in $\mathbb R^n$. Combining \eqref{eq:asymptotic-f} and \eqref{ineq:subcritical-6} yields $-\Delta_p^H v=\tau v^\alpha_0$ in $\mathbb R^n$, which contradicts Theorem \ref{thm:subcritical} and Corollary \ref{cor:p>=n}, so \eqref{eq:asymptotic-estimate 1} holds.
\end{proof}

A similar argument yields the following corollary.
\begin{corollary}\label{cor:asymptotic-2}
    Let $0<p-1<\alpha_0<p^*-1$, and let $\Omega\subsetneq\mathbb R^n$ be a domain. Then there exists a positive constant $C=C(n,p,\alpha_0)$ such that any nonnegative solution $u$ to $-\Delta_p^H u=u^{\alpha_0}$ in $\Omega$, the following holds:
    \begin{equation}\label{ineq:subcritical-7}
        u(x)^{\frac{\alpha_0-p+1}{p}}+|\nabla u(x)|^{\frac{\alpha_0-p+1}{\alpha_0+1}}\leqslant Cd^{-1}(x),\quad\forall x\in\Omega,
    \end{equation}
    In particular, if $\Omega$ is an exterior domain, i.e., $\{|x|>R_0\}\subset \Omega$ for some $R_0>0$, then
    \begin{equation*}
        u(x)^{\frac{\alpha_0-p+1}{p}}+|\nabla u(x)|^{\frac{\alpha_0-p+1}{\alpha_0+1}}\leqslant C|x|^{-1},\quad\forall x\in B_{2R_0}^c.
    \end{equation*}
\end{corollary}
\begin{proof}
    We adopt the notation from the proof of Theorem \ref{thm:asymptotic-1}. If estimate \eqref{ineq:subcritical-7} fails, then there exist $\Omega_k$, $u_k$, and $y_k\in \Omega_k$ such that $-\Delta_p^H u_k=u_k^{\alpha_0}$ in $\Omega_k$ and $M_k(y_k)>2k d^{-1}(y_k)$.
    By the doubling lemma, we can deduce \eqref{ineq:subcritical-4} and \eqref{ineq:subcritical-5}. Moreover, $v_k$ satisfies $-\Delta_p^H v_k =v_k^{\alpha_0}$ on $\overline{B_k}$. Similarly, we have $v_k \in C^{1,q}_{\mathrm{loc}}(B_k)$ for some $q\in (0,1)$. Thus, there exists a subsequence $v_{k_j}$ that converges in $C^{1,q'}_{\mathrm{loc}}(\mathbb R^n)$ to a positive function $v$ satisfying $-\Delta_p^H v=v^{\alpha_0}$ in $\mathbb R^n$, which contradicts Theorem \ref{thm:subcritical} and Corollary \ref{cor:p>=n}, so \eqref{ineq:subcritical-7} holds.
\end{proof}
\section{Supercritical case: Liouville theorem for stable solutions}\label{sec:stable}

In this section, we study the stable solutions and solutions that are stable outside a compact set to the anisotropic $p$-Laplace equation \eqref{eq:anisotropic}, and prove that these solutions must be trivial under certain assumptions on $n$, $p$, and $f$. Assume $u$ is a solution to \eqref{eq:anisotropic}. By Lemma \ref{lem:regularity}, $u\in C^{1,q}_\mathrm{loc}(\mathbb{R}^n)$ and $a(\nabla u)\in H^1_{\mathrm{loc}}(\mathbb{R}^n)$ for some $q\in (0,1)$.

We first prove Theorem \ref{thm:stable}.
\begin{proof}[Proof of Theorem \ref{thm:stable}]
    Taking $\varphi=|u|^{\delta-1} u \eta^{\gamma}$ in \eqref{weak-sol}, where $\delta \geqslant 1$ and $\gamma > 0$ are to be determined, we obtain
    \begin{equation}\label{ineq:stable-1}
        \int |u|^{\delta-1}uf(u)\eta^\gamma =\delta \int |u|^{\delta-1}H^p(\nabla u) \eta^\gamma+\frac \gamma p\int|u|^{\delta-1}u\nabla_{\xi} H^p(\nabla u)\cdot\eta^{\gamma-1}\nabla\eta.
    \end{equation}
    By using Young's inequality, we can deduce
    \begin{equation}\label{ineq:stable-2}
        \frac \gamma p\left|\int|u|^{\delta-1}u\nabla_{\xi} H^p(\nabla u)\cdot\eta^{\gamma-1}\nabla\eta\right|\leqslant \varepsilon \int |u|^{\delta-1}H^p(\nabla u) \eta^\gamma +C \int |u|^{\delta+p-1}\eta^{\gamma-p} |\nabla \eta|^{p}.
    \end{equation}
    By substituting \eqref{ineq:stable-2} into \eqref{ineq:stable-1} and using $u^2f'(u)\geqslant\alpha_0uf(u)$, we conclude
    \begin{equation}\label{ineq:stable-3}
        \alpha_0(\delta-\varepsilon)\int|u|^{\delta-1}H^p(\nabla u)\eta^\gamma\leqslant\int|u|^{\delta+1}f'(u)\eta^\gamma+ C \int |u|^{\delta+p-1}\eta^{\gamma-p} |\nabla \eta|^{p}.
    \end{equation}
    
    Taking $\varphi=|u|^{\frac{\delta-1}{2}}u\eta^{\frac{\gamma}{2}}$ in $\left<\mathcal{L}_u\varphi,\varphi\right>_{L^2}\geqslant 0$ and applying Young's inequality, we obtain
    \begin{align*}
        &\int |u|^{\delta+1}f'(u)\eta^\gamma \leqslant \int \frac{1}{p} \nabla_{\xi_i \xi_j}^2 H^p(\nabla u) \partial_i \left(|u|^{\frac{\delta-1}{2}}u\eta^{\frac{\gamma}{2}}\right) \partial_j \left(|u|^{\frac{\delta-1}{2}}u\eta^{\frac{\gamma}{2}}\right)\\
        \leqslant ~&\left(\frac{(\delta+1)^2}{4}+\varepsilon\right) \int|u|^{\delta-1} \frac{1}{p} \nabla_{\xi_i \xi_j}^2H^p(\nabla u)u_iu_j\eta^\gamma +C\int |u|^{\delta+1}\eta^{\gamma-2} |\nabla \eta|^2 H^{p-2}(\nabla u)\\
        \leqslant ~&\left(\frac{(\delta+1)^2}{4}(p-1)+p\varepsilon\right)\int |u|^{\delta-1}H^p(\nabla u)\eta^\gamma +C\int |u|^{\delta+p-1}\eta^{\gamma-p} |\nabla \eta|^p.
    \end{align*}
    Observe that the case $p = 2$ is special, as the second inequality holds directly. By substituting this into \eqref{ineq:stable-3}, we obtain
    \begin{align*}
        \left(\alpha_0\delta-\frac{(\delta+1)^2}{4}(p-1)-(\alpha_0+p)\varepsilon\right)\int |u|^{\delta-1}H^p(\nabla u)\eta^\gamma \leqslant C\int |u|^{\delta+p-1}\eta^{\gamma-p} |\nabla \eta|^p.
    \end{align*}

    Let $\delta_0=\frac{2\alpha_0-(p-1)+2\sqrt{\alpha_0[\alpha_0-(p-1)]}}{p-1}$. Since $\alpha_0>p-1$, $\delta_0>1$. For any $\delta \in[1, \delta_0)$, it follows that $\alpha_0\delta>\frac{(\delta+1)^2}{4}(p-1)$. Consequently, when $\varepsilon$ is sufficiently small, by combining the above inequality with \eqref{ineq:stable-1} and \eqref{ineq:stable-2}, it follows that
    \begin{align}\label{ineq:stable-4}
        \int |u|^{\delta-1}uf(u)\eta^{\gamma} +\int |u|^{\delta-1}H^p(\nabla u)\eta^\gamma \leqslant C\int |u|^{\delta+p-1}\eta^{\gamma-p} |\nabla \eta|^p.
    \end{align}

    Define $\mathcal{X}_1=\{u\leqslant M\}$, $\mathcal{X}_2=\{u>M\}$. By the conditions of $f$, $uf(u)\geqslant C|u|^{\beta_0+1}$ on $\mathcal{X}_1$ and $uf(u)\geqslant C|u|^{\alpha_0+1}$ in $\mathcal{X}_2$. By taking $\gamma = \frac{(\alpha_0+\delta)p}{\alpha_0-p+1}$ in \eqref{ineq:stable-4}, we obtain
    \begin{align*}
        &\int_{\mathcal{X}_1}|u|^{\beta_0+\delta}\eta^\gamma+\int_{\mathcal{X}_2}|u|^{\alpha_0+\delta}\eta^\gamma \leqslant C \int |u|^{\delta+p-1}\eta^{\gamma-p} |\nabla \eta|^p\\
        \leqslant~&\varepsilon \int_{\mathcal{X}_1} |u|^{\beta_0+\delta}\eta^\gamma+C \int_{\mathcal{X}_1} \eta^{\gamma-\frac{(\beta_0+\delta)p}{\beta_0-p+1}}|\nabla \eta|^{\frac{(\beta_0+\delta)p}{\beta_0-p+1}}\\
        &+\varepsilon \int_{\mathcal{X}_2} |u|^{\alpha_0+\delta}\eta^\gamma+C\int_{\mathcal{X}_2}\eta^{\gamma-\frac{(\alpha_0+\delta)p}{\alpha_0-p+1}}|\nabla \eta|^{\frac{(\alpha_0+\delta)p}{\alpha_0-p+1}}\leqslant R^{n-\frac{(\beta_0+\delta)p}{\beta_0-p+1}}.
    \end{align*}
    In the final inequality, we use the fact that $n-\frac{(\alpha_0+\delta)p}{\alpha_0-p+1}\leqslant n-\frac{(\beta_0+\delta)p}{\beta_0-p+1}$. Under conditions (1) and (2) of the theorem, we have $n-\frac{(\beta_0+\delta_0)p}{\beta_0-p+1}<0$. Therefore, we choose $\delta\in[1,\delta_0)$ such that $n-\frac{(\beta_0+\delta)p}{\beta_0-p+1}<0$. The conclusion follows by letting $R\to\infty$.
\end{proof}

In what follows, we apply Pohozaev identity (Theorem 1.3 in \cite{Montoro2023-Pohozaev}) to prove Theorem \ref{thm:stable-outside}.

\begin{proof}[Proof of Theorem \ref{thm:stable-outside}]
    Let $R_0>0$ be such that $\mathcal{K}\subset B_{R_0}$. The proof proceeds in three steps.
    
    \textbf{Step 1.} We prove that
    \begin{equation}\label{eq:integrability-stable}
        \int_0^u f(s)\mathrm{d}s\in L^1(\mathbb{R}^n),\quad  uf(u)\in L^1(\mathbb{R}^n),\quad\text{and} \quad H^p(\nabla u)\in L^1(\mathbb{R}^n).
    \end{equation}

    Since $u$ is a stable solution to \eqref{eq:anisotropic} on  $B_{R_0}^c$, similar to \eqref{ineq:stable-4} with $\delta=1$, the following inequality holds:
    \begin{equation}\label{ineq:stable-5}
        \int uf(u)\tilde\eta^{\gamma} +\int H^p(\nabla u) \tilde\eta^\gamma \leqslant C\int |u|^{p}\tilde\eta^{\gamma-p}|\nabla\tilde\eta|^p,
    \end{equation}
    for any cut-off function $\tilde\eta\in C_c^1(B_{R_0}^c)$ with $0\leqslant\tilde\eta\leqslant 1$, 

    To deal with $|u|^{p}$, we define $\mathcal{X}_1=\{|u|\leqslant M\}$ and $\mathcal{X}_2=\{|u|>M\}$. It follows from \eqref{eq:f-outside-p>=n} that
    \begin{equation}\label{ineq:stable-6}
        |u|^{p}\leqslant C
        \begin{cases}
            [uf(u)]^{\frac{p}{\beta_0+1}},&\text{on }\mathcal{X}_1,\\
            [uf(u)]^{\frac{p}{\alpha_0+1}},&\text{in }\mathcal{X}_2.
        \end{cases}
    \end{equation}
    By taking $\gamma=\frac{(\alpha_0+1)p}{\alpha_0-p+1}$ and combining inequalities \eqref{ineq:stable-5} and \eqref{ineq:stable-6}, we obtain
    \begin{align}
        \begin{split}\label{ineq:stable-7}
            &\int uf(u)\tilde\eta^{\gamma}+\int H^p(\nabla u) \tilde\eta^\gamma\\
            \leqslant~&C\int_{\mathcal{X}_1}[uf(u)]^{\frac{p}{\beta_0+1}}\tilde\eta^{\gamma-p} |\nabla\tilde\eta|^{p}+C\int_{\mathcal{X}_2}[uf(u)]^{\frac{p}{\alpha_0+1}} \tilde\eta^{\gamma-p} |\nabla\tilde\eta|^{p}\\
            \leqslant~&\frac{1}{2}\int_{\mathcal{X}_1} uf(u)\tilde\eta^{\gamma}+C\int_{\mathcal{X}_1} \tilde\eta^{\gamma-\frac{(\beta_0+1)p}{\beta_0-p+1}}|\nabla \tilde\eta|^{\frac{(\beta_0+1)p}{\beta_0-p+1}}+\frac{1}{2}\int_{\mathcal{X}_2}uf(u)\tilde\eta^{\gamma} +C\int_{\mathcal{X}_2} \tilde\eta^{\gamma-\frac{(\alpha_0+1)p}{\alpha_0-p+1}}|\nabla \tilde\eta|^{\frac{(\alpha_0+1)p}{\alpha_0-p+1}}\\
            \leqslant~&C\int|\nabla \tilde\eta|^{\frac{(\beta_0+1)p}{\beta_0-p+1}}.
        \end{split}
    \end{align}
    Here, we use the fact that $\frac{\beta_0+1}{\beta_0-p+1}\leqslant \frac{\alpha_0+1}{\alpha_0-p+1}$.

    For any $r>R_0+3$, we define a cutoff function $\tilde\eta_r\in C_c^{\infty}(B_{2r} \backslash B_{R_0+1})$ such that
    $$\tilde\eta_r=\begin{cases}
        0,&\text{if }x\in B_{R_0+1},\\
        1,&\text{if }x\in B_r\backslash B_{R_0+2},\\
        0,&\text{if }x\in B_{2r}^c.
    \end{cases}$$
    By substituting $\tilde\eta=\tilde\eta_r$ into \eqref{ineq:stable-7}, we deduce
    $$\int_{B_r\backslash B_{R_0+2}} \left[\int_0^u f(s)\mathrm{d}s+ H^p(\nabla u)\right] \leqslant \int_{B_r\backslash B_{R_0+2}}\left[f(u) u + H^p(\nabla u)\right] \leqslant C\left(1+r^{n-\frac{(\beta_0+1)p}{\beta_0-p+1}}\right) \leqslant C.$$
    As $r\to\infty$, \eqref{eq:integrability-stable} holds. 
    Moreover, $u\in L^{\beta_0+1}(\mathcal{X}_1)\cap L^{\alpha_0+1}(\mathcal{X}_2)$ by \eqref{ineq:stable-6}.

    \textbf{Step 2.}
    By applying Pohozaev identity in the anisotropic setting (see Theorem 1.3 in \cite{Montoro2023-Pohozaev}), we obtain 
    \begin{equation}\label{eq:Pohozaev }
        n\int_{\mathbb{R}^n}\mathrm{d}x\int_{0}^{u}f(s)\mathrm{d}s=\frac{n-p}{p}\int_{\mathbb{R}^n} H^p(\nabla u)\mathrm{d}x.
    \end{equation}
    When $p\geqslant n$, we have $u \equiv 0$. 

    \textbf{Step 3.} Now we consider the case $2\leqslant p<n$ and $p-1<\alpha_0\leqslant \beta_0<p^*-1$. It follows from \eqref{eq:f-outside-p<n} that
    $$\beta_0\int_0^u f(s)\geqslant\int_0^u sf'(s)=uf(u)-\int_0^u f(s),\quad\forall u\ne 0,$$
    that is
    \begin{equation}\label{ineq:stable-8}
        \int_{\mathbb{R}^n}\mathrm{d}x\int_0^u f(s)\mathrm{d}s\geqslant\frac{1}{\beta_0+1}\int_{\mathbb{R}^n}uf(u)\mathrm{d}x.
    \end{equation}

    By testing $u\eta$ in \eqref{eq:anisotropic}, we obtain
    $$\int uf(u)\eta=\int H^p(\nabla u)\eta +\int uH^{p-1}(\nabla u) \nabla_{\xi}H(\nabla u) \nabla\eta\triangleq \operatorname{I}+\operatorname{II}.$$
    From \eqref{ineq:stable-6} and the integrability conditions \eqref{eq:integrability-stable}, we derive
    \begin{align*}
        |\operatorname{II}|
        \leqslant~&\left(\int_{\mathcal{W}\cap \mathcal{X}_1}H^{p} \right)^{\frac{1}{p}} \left(\int_{\mathcal{W}\cap\mathcal{X}_1} |u|^{\beta_0+1}\right)^{\beta_0+1}\left(\int_{\mathcal{W}\cap \mathcal{X}_1}|\nabla\eta|^{\frac{(\beta_0+1)p}{\beta_0+1-p}} \right)^{\frac{\beta_0-p+1}{(\beta_0+1)p}} \\
        &+\left(\int_{\mathcal{W}\cap \mathcal{X}_2}H^{p} \right)^{\frac{1}{p}} \left(\int_{\mathcal{W}\cap\mathcal{X}_2}|u|^{\alpha_0+1}\right)^{\alpha_0+1}\left(\int_{\mathcal{W}\cap\mathcal{X}_2} |\nabla\eta|^{ \frac{(\alpha_0+1)p}{\alpha_0-p+1}} \right)^{\frac{\alpha_0-p+1}{(\alpha_0+1)p}}\to 0,\quad\text{as }R\to\infty,
    \end{align*}
    where $\mathcal{W}=B_{2R}\backslash B_R$. Hence
    \begin{equation}\label{ineq:stable-9}
        \int_{\mathbb{R}^n}uf(u)  =\int_{\mathbb{R}^n}H^p(\nabla u).
    \end{equation}
    By combining \eqref{eq:Pohozaev }, \eqref{ineq:stable-8} and \eqref{ineq:stable-9}, we conclude
    $$\frac{n-p}{np} \int_{\mathbb{R}^n} f(u)u\mathrm{d}x = \frac{n-p}{np}\int_{\mathbb{R}^n} H^p(\nabla u)\mathrm{d}x = \int_{\mathbb{R}^n}\mathrm{d}x\int_{0}^{u}f(s)\mathrm{d}s\geqslant \frac{1}{1+\beta_0} \int_{\mathbb{R}^n} uf(u)\mathrm{d}x.$$
    This contradicts $\beta_0<p^*-1$, thereby implying $u\equiv 0$.
\end{proof}

\textbf{Acknowledgments.} The authors would like to thank Professor Xi-Nan Ma for the constant encouragement in this article.

\printbibliography[heading=bibintoc, title=\ebibname]

\appendix

\end{document}